\documentclass[a4paper]{amsart}
\usepackage{amsmath,amsthm,amssymb}
\usepackage[dvipdfmx]{graphicx}

\newtheorem{thm}{Theorem}[section]
\newtheorem{lem}[thm]{Lemma}
\newtheorem{cor}[thm]{Corollary}
\newtheorem{prop}[thm]{Proposition}

\theoremstyle{definition}

\newtheorem{defn}[thm]{Definition}

\theoremstyle{remark}

\newtheorem{rem}[thm]{Remark}        

\numberwithin{equation}{section}
\newcommand{\B}[2]{B_{#1}(#2)}
\newcommand{\Pro}{\mathcal{P}}
\newcommand{\R}{\mathbb{R}}

\usepackage[abbrev]{amsrefs}

\makeatletter
\def\@makefnmark{%
\leavevmode
\raise.9ex\hbox{\check@mathfonts
\fontsize\sf@size\z@\normalfont%
\@thefnmark}%
}
\makeatother

\title[Lower bound of coarse Ricci curvature and eigenvalues]{Lower bound of coarse Ricci curvature on metric measure spaces and eigenvalues of Laplacian}
\author{Yu Kitabeppu}
\thanks{Partly supported by the Grant-in-Aid for JSPS Fellows, The Ministry of Education, Culture, Sports, Science and Technology, Japan}
\begin{document}
\maketitle
 \begin{abstract} 
 In this paper, we investigate the coarse Ricci curvature on metric spaces. 
 We prove that a Bishop-Gromov 
 inequality gives a lower bound of coarse Ricci curvature. 
 The lower bound does not coincide with the constant 
 corresponding to curvature in Bishop-Gromov inequality. 
 As a corollary, we obtain a lower bound of coarse Ricci curvature for a metric space satisfying 
 the curvature-dimension condition.
 Moreover we give an important example, Heisenberg group, 
 which does not satisfy the curvature-dimension condition 
 for any constant but has a lower bound of coarse Ricci curvature. 
 We also have an estimate of the eigenvalues of the Laplacian by a lower bound of coarse Ricci 
 curvature.  
 \end{abstract}
   
 \section{Introduction}
  In this paper, we investigate coarse Ricci curvature on metric spaces. 
  Ollivier \cite{O1} defined a notion of \emph{coarse Ricci curvature}. 
  He obtained a relation between the coarse Ricci curvature 
  and the ordinary Ricci curvature on a Riemannian manifold (see \cite{O1}). On the other hand, 
  Lott, Villani and Sturm introduced the curvature-dimension condition, 
  which is a generalization of a Ricci curvature bounded by below 
  for geodesic metric spaces \cite{LV,St1,St2}. 
  However, no relation of these two definitions have been known. 
  One of our motivation is to find such a relation for 
  general metric measure spaces. 
  In this paper we prove that a Bishop-Gromov inequality (see Definition \ref{BGineq}) 
  gives a lower bound of the coarse 
  Ricci curvature. Note that the curvature-dimension condition implies the Bishop-Gromov inequality
  \cite{LV,St2}. 
  We also prove that the Heisenberg group with the sub-Riemannian metric is an example which 
  does not satisfy the curvature-dimension condition but has a lower bound of coarse Ricci curvature. 
  As another result, we establish an eigenvalue estimate for the Laplacian 
  by a lower bound of the coarse Ricci curvature. 
  For graphs, more precise estimates of a lower bound of 
  the coarse Ricci curvature and the eigenvalue of Laplacian 
  have already been given
  by Lin-Yau, Jost-Bauer-Liu, Jost-Liu \cite{LYY,BJL,JL}. 
  Our eigenvalue estimate is for general metric spaces and for not only positive lower bound but 
  also lower bound of any real number. 
  
  Let $(X,d)$ be a complete separable metric space. 
  We denote by $\Pro(X)$ the set of all Borel probability measures on $X$. 
  We call a family $\{m_x\}_{x\in X}$ of measures in $\Pro(X)$ \emph{a random walk}
   (see Section 2 for the precise definition). 
  The following is defined by Ollivier \cite{O1}. 
   \begin{defn}[Coarse Ricci curvature]
    For any two points $x,y\in X$, the \emph {coarse Ricci curvature} 
    $\kappa(x,y)$ associated with $\{m_x\}_{x\in X}$ along $xy$ is defined by
     \begin{align}
     \label{cRiccidef}
      \kappa(x,y):=1-\frac{W_1(m_x,m_y)}{d(x,y)},
     \end{align}
    where $W_1$ is the $L^1$-Wasserstein metric.
   \end{defn}
   It is clear that $\kappa(x,y)\leq 1$ for any $x,y\in X$. 
   
   In this paper we call a complete separable metric space with a positive Radon measure 
   \emph{a metric measure space}. 
   We consider the following random walk. 
   \begin{itemize}
    \item \emph{The $r$-step random walk} :  
             \begin{align}
              m^r_x=\frac{\chi_{B_r(x)}}{\nu(B_r(x))}\nu,\notag
             \end{align}
             on a metric measure space $(X,d,\nu)$, 
             where $\chi_{B_r(x)}$ denotes the characteristic function of $B_r(x)$. 
             In this paper, we denote the open ball 
             of radius $r>0$ and centered at $x\in X$ by $B_r(x)$.
   \end{itemize}
   We define an important notion.
   \begin{defn}
   \label{BGineq}
   For two real numbers $K$ and $N>1$, 
   we define a function $s_{K,N} : [0,\infty)\rightarrow \mathbb{R}$, by 
   \begin{align}
     s_{K,N}(t):=\begin{cases}
                          \sqrt{(N-1)/K}\sin (t\sqrt{K/(N-1)}) & \text{if}\; K>0,\\
                          t & \text{if}\; K=0,\\
                          \sqrt{(N-1)/-K}\sinh (t\sqrt{-K/(N-1)}) & \text{if}\; K<0.
                         \end{cases}\notag
   \end{align}
   A metric measure space $(X,d,\nu)$ satisfies 
   \emph{the Bishop-Gromov inequality $[BG_{K,N}]$} if 
   \begin{align}
    \frac{\nu\left(B_R(x)\right)}{\nu\left(B_r(x)\right)}
    \leq \frac{\int_0^R s_{K,N}(t)^{N-1}\, dt}{\int_0^r s_{K,N}(t)^{N-1}\, dt}\label{BGineq}
   \end{align}
   holds for any $x\in X$ and for any $0<r<R\leq \pi\sqrt{(N-1)/\max {\{K,0\}}}$ with 
   the convention $1/0=\infty$.
   \end{defn}
   A metric space $(X,d)$ is a geodesic metric space if for any $x,y\in X$, there exists 
   a curve $\gamma :[0,1]\rightarrow X$ joining $x$ to $y$ such that  
   $d(\gamma(s),\gamma(t))=|s-t|d(x,y)$ for any $s,t\in [0,1]$. 
  The following theorem is one of the main results in this paper.
  \begin{thm}
  \label{main1}
   Let $(X,d,\nu)$ be a geodesic metric measure space satisfying 
   $[BG_{K,N}]$ for two real numbers $K$ and $N>1$. 
   Then, for any $0<r<\pi\sqrt{(N-1)/\max{\{K,0\}}}$, the coarse Ricci curvature associated with 
   the $r$-step random walk satisfies 
   \begin{align}
    \inf_{x,y\in X}\kappa(x,y)\geq 1-2r\frac{s_{K,N}(r)^{N-1}}{\int_0^r s_{K,N}(t)^{N-1}\, dt}.\label{cor1eq1}
   \end{align}
  \end{thm}
  \begin{rem}
   Ollivier \cite {O1} obtained a relation between the coarse Ricci curvature associated 
   with the $r$-step random walk 
   and  the ordinal Ricci curvature in the Riemannian case. However 
   he had only an asymptotic estimates as $r$ tends to zero. In particular, his estimate 
   gives no lower bound of the coarse Ricci curvature if the manifold is noncompact. 
   Theorem \ref{main1} gives a priori estimate for each $r>0$. 
  \end{rem}
  \begin{rem}
   For any $r>0$, calculating the right-hand side of (\ref{cor1eq1}), 
   we have $\inf_{x,y\in X}\kappa(x,y)\geq 1-2N$ on a metric measure space 
   $X$ which satisfies $[BG_{0,N}]$. 
   In the case of $[BG_{K,N}]$, we obtain
   \begin{align}
    \lim_{r\rightarrow 0}\inf_{x,y\in X}\kappa(x,y)\geq 1-2N.\notag
   \end{align} 
   \end{rem}
   According to \cite{J}, the $n$-dimensional \emph{Heisenberg group $\mathbb{H}^n$} 
   with left invariant sub-Riemannian metric 
   does not satisfy $[CD_{K,N}]$ for any $K$ and $N>1$. 
   Nevertheless $\mathbb{H}^n$ satisfies $[BG_{0,2n+3}]$ (see \cite{J,Oh}), 
   which together with Theorem \ref{main1} implies the following.
   \begin{cor}
   \label{heisenberg}
    The coarse Ricci curvature 
    associated with the $r$-step random walk on the $n$-dimensional Heisenberg group 
    $\mathbb{H}^n$ satisfies 
    \begin{align}
     \inf_{x,y\in \mathbb{H}^n}\kappa(x,y)\geq 1-2(2n+3).\notag
    \end{align}
   \end{cor}
   \begin{rem}
   Corollary \ref{heisenberg} says that even if 
   the coarse Ricci curvature is bounded by below, the curvature-dimension condition 
   does not hold in general.  
   \end{rem}
    
  To state another result, we define a version of Laplacian. Let $(X,d,\{m_x\}_{x\in X})$ be 
  a complete separable 
  metric space with a random walk.  
   \begin{defn}[Laplacian]
     We define \emph{the Laplacian of a function} $f:X\rightarrow \R$ by
     \begin{align}
      \Delta f(x):=f(x)-\int_X f(y)\, dm_x(y)
     \end{align}
    for $x\in X$, whenever the right-hand side is defined.
    \end{defn}
  We consider the eigenvalue problem for the Laplacian on $(X,d,\{m_x\}_{x\in X})$.  
  The following result holds even for discrete spaces and even for a general random walk.  
   \begin{thm}
   \label{main2}
   Assume that 
   $\kappa(x,y)\geq \kappa$ for any $x,y\in X$ and for a constant $\kappa$.
   Then we have the following.
     \begin{enumerate}
      \item If there exists a nonconstant Lipschitz function $f$ on $X$ such that $\Delta f=\lambda f$ for 
      a real number $\lambda$, 
      then it turns out that $\kappa\leq\lambda\leq 2-\kappa$.\\
      \item If $\kappa$ is positive and if a Lipschitz function $f$ on $X$ satisfies $\Delta f=0$, 
      then $f$ is a constant function.
     \end{enumerate}
   \end{thm}
   Theorem \ref{main2}$(2)$ is a Liouville type theorem. 
   Moreover, if $X$ is compact, we have the following. 
   \begin{cor}
   \label{cor2}
    Let $(X,d,\nu)$ be a compact geodesic metric measure space with a lower bound, 
    say $\kappa$, of the coarse Ricci curvature associated with the $r$-step random walk. 
    Assume that $X$ satisfies a Bishop-Gromov inequality, 
    then any eigenvalue $\lambda$ of the Laplacian satisfies 
    \begin{align}
     \kappa\leq \lambda\leq 2-\kappa. \label{cor2eq}
    \end{align} 
   \end{cor}
   \begin{rem}
    Ollivier proved a similar claim in his paper \cite{O1}, for which  
    he assumed a positivity of coarse Ricci curvature and 
    a reversibility of invariant distribution. 
   \end{rem}
 \section{Definition and Basic Properties of Coarse Ricci Curvature}
  Let $(X,d)$ be a separable metric space and $\Pro(X)$ the set of all Borel probability measures 
  on $X$. 
  For $\mu,\nu\in\Pro(X)$, 
  a measure $\xi\in\Pro(X\times X)$ is called \emph{a coupling between $\mu$ and $\nu$} 
  if $\xi$ satisfies $\xi(A\times X)=\mu(A)$ and $\xi(X\times B)=\nu(B)$ 
  for any Borel sets $A,B\subset X$. $\Pi(\mu,\nu)$ denotes the set of all couplings between 
  $\mu$ and $\nu$. 
  Let $\Pro_p(X)$ for $p\geq 1$ be the set of all Borel probability measures 
  which have finite $p$-th moment, 
  where the $p$-th moment of $\mu\in\Pro(X)$ is defined to be 
  $\int_Xd(x,o)^p\,d\mu(x)$ for a point $o$ 
  in $X$.
  It is easy to check that $\Pro_p(X)\subset \Pro_q(X)$ for $p>q$. 
  We define \emph{the $L^p$-Wasserstein metric $W_p$} on $\Pro_p(X)$ by 
   \begin{align}
    W_p(\mu,\nu):=\inf_{\xi\in\Pi(\mu,\nu)}\left\{\int_{X\times X}d(x,y)^p\, d\xi(x,y)\right\}^{1/p}
   \end{align}
  for $\mu,\nu\in\Pro_p(X)$. 
  The metric space $(\Pro_p(X), W_p)$ is called \emph{the $L^p$-Wasserstein space}. 
  The following is a well-known 
  result (see \cite{V1,V2}).
   \begin{lem}[Kantorovich-Rubinstein duality] 
    For any $\mu,\nu\in \Pro_1(X)$ we have 
     \begin{align}
     \label{Kduality}
      W_1(\mu,\nu)=\sup_{f\in Lip_1(X)}\left(\int_Xf\,d\mu-\int_Xf\,d\nu\right),
     \end{align}
    where $Lip_1(X)$ is the set of all 1-Lipschitz functions on $X$.
   \end{lem}
  We define \emph{a random walk} on $X$ as a family $\{m_x\}_{x\in X}$ of Borel probability 
  measures on $X$ such that
    the map $x\mapsto m_x$ is a Borel measurable map from $(X,d)$ to $(\Pro_1(X), W_1)$. 
 The following proposition is needed for the proof of Theorem \ref{main1}. 
   \begin{prop}[\cite{O1}]
   \label{geodproperty}
    Let $(X,d)$ be a geodesic metric space, $\{m_x\}_{x\in X}$ a random walk and $K\in\R$. 
    Assume that there exists $\epsilon>0$ such that $\kappa(x,y)\geq K$ holds for any $x,y\in X$ 
    with $d(x,y)<\epsilon$. Then we have $\inf_{x,y\in X}\kappa(x,y)\geq K$.
   \end{prop}
    \begin{rem}
     If there exists a limit $\lim_{y\rightarrow x}\kappa(x,y)$ and if the limit is 
     bounded below by $\kappa$ for any $x\in X$, then we have 
     $\inf_{x,y\in X}\kappa(x,y)\geq \kappa$ by Proposition \ref{geodproperty}.
    \end{rem}
 \section{Proof of Main Theorems}
 We define a function $F : \mathbb{R}_{\geq 0}\rightarrow \mathbb{R}_{\geq 0}$ by 
 $F(r):=\int_0^r s_{K,N}(t)^{N-1}\, dt$ where $\mathbb{R}_{\geq 0}$ denotes the set of 
 nonnegative real numbers. 
  We need the following lemma to prove Theorem \ref{main1}.
  \begin{lem}
  \label{lem2} 
   Let $(X,d,\nu)$ be a geodesic metric measure space satisfying 
   $[BG_{K,N}]$. For given $r>0$ we take two points $x,y\in X$ such that $d(x,y)<r$. 
 We have
   \begin{align}
    \frac{\nu\left(B_r(x)\setminus B_r(y)\right)}{\nu\left(B_r(x)\right)}\leq
    1-\frac{F(r-d(x,y)/2)}{F(r+d(x,y)/2)}=\frac{F'\!(r-d(x,y)/2)}{F(r+d(x,y)/2)}d(x,y)+o(d(x,y))\label{lemeq2}
   \end{align}
   as $d(x,y)\rightarrow 0$.
  \end{lem}
  \begin{proof}
  We prove (\ref{lemeq2}). 
  Since $X$ is a geodesic metric space, there exists $z\in X$ such that 
  $d(x,y)/2=d(x,z)=d(y,z)$. Setting $d:=d(x,y)$, we have 
  $\B{r-d/2}{z}\subset\B{r}{x}\cap\B{r}{y}\subset\B{r}{x}\subset \B{r+d/2}{z}$. 
  By the condition $[BG_{K,N}]$, we get 
  \begin{align}
   \nu(\B{r}{x})\leq \nu(\B{r+d/2}{z})\leq \frac{F(r+d/2)}{F(r-d/2)}\nu(\B{r-d/2}{z}).\notag
  \end{align}
  Hence we have 
  \begin{align}
   \nu(\B{r}{x}\setminus\B{r}{y})&=\nu(\B{r}{x})-\nu(\B{r}{x}\cap\B{r}{y})\notag\\
      &\leq \nu(\B{r}{x})-\nu(\B{r-d/2}{z})\notag\\
      &\leq \left\{1-\frac{F(r-d/2)}{F(r+d/2)}\right\}\nu(\B{r}{x})\notag
  \end{align}
  which implies the inequality in (\ref{lemeq2}). Since $F$ is a $C^1$ function, we have 
 \begin{align}
  F\!\left(r+\frac{d}{2}\right)=F\!\left(r-\frac{d}{2}\right)+F'\!\!\left(r-\frac{d}{2}\right)d+o\!\left(d\right)\notag
 \end{align}
 as $d\rightarrow 0$. This completes the proof. 
  \end{proof}
   \begin{proof}[Proof of Theorem \ref{main1}]
   Take any two points $x,y\in X$ with 
   $d(x,y)<r$. Without loss of generality, we may assume $\nu(\B{r}{x})\geq \nu(\B{r}{y})$. 
    By the assumption, $\nu(\B{r}{x}\cap\B{r}{y})/\nu(\B{r}{x})\leq \nu(\B{r}{x}\cap\B{r}{y})/\nu(\B{r}{y})$. 
    By Lemma \ref{lem2} we have
    \begin{align}
      m_x\left(B_r(x)\setminus B_r(y)\right)\leq \frac{F'(r-d/2)}{F(r+d/2)}d+o(d),
    \end{align}
    where $d:=d(x,y)$. 
    
    \begin{figure}[h]
     \begin{minipage}{55mm}
       \begin{center}
    \includegraphics[width=5.5cm]{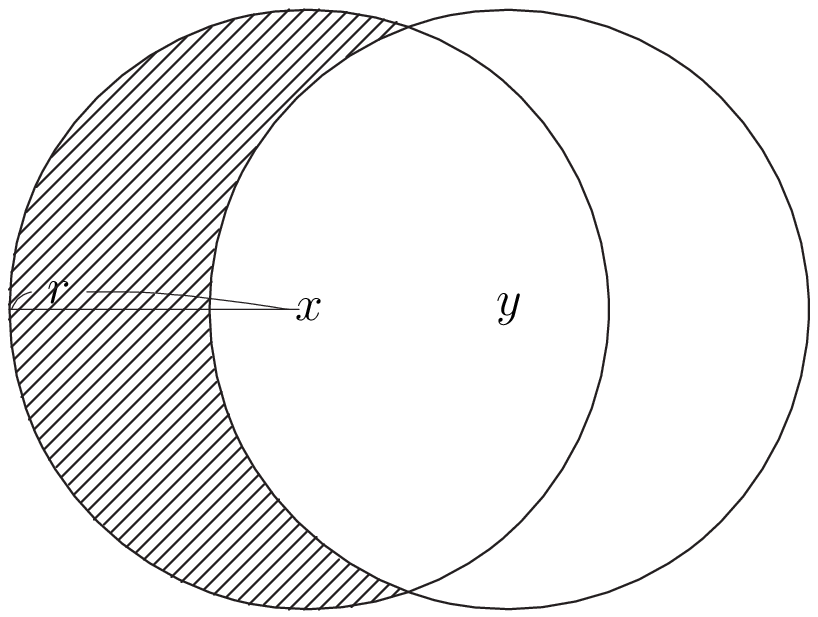}
       \end{center}
       \caption{\footnotesize{$m_x(B_r(x)\setminus B_r(y))$}}
       \end{minipage}
    \begin{minipage}{65mm}
     \begin{center}
     \includegraphics[width=5.5cm]{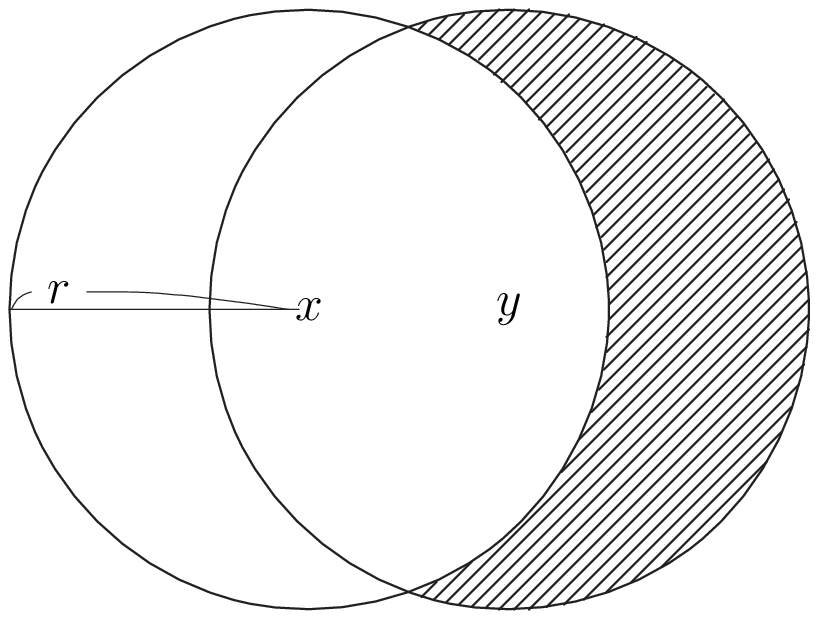}
     \end{center}
     \caption{\footnotesize{$m_y(B_r(y)\setminus B_r(x))$}}
    \end{minipage}
    \end{figure}
    
    We consider the following transportation plan. 
    We transport $m_x$ to $m_y$. 
    Since the mass in $B_r(x)\cap B_r(y)$ measured by $m_x$ does not have to move, 
    we simply transport the mass in $B_r(x)\setminus B_r(y)$ to 
    $B_r(y)\setminus B_r(x)$. 
    In general, this plan is far from an optimal one. 
    We have 
     \begin{align}
      W_1(m_x,m_y)
      &\leq m_x\left(B_r(x)\setminus B_r(y)\right)\cdot (d+2r)\notag\\
      &\leq(d+2r)\left\{\frac{F'(r-d/2)}{F(r+d/2)}d+o(d)\right\}\notag
     \end{align}
     as $d\rightarrow 0$. 
     Then, by the definition of the coarse Ricci curvature and by the continuity of $F$, 
     \begin{align}
      \kappa(x,y)&=1-\frac{W_1(m_x,m_y)}{d(x,y)}\label{main1pfeq}\\
      &\geq 1-(d+2r)\left\{\frac{F'(r-d/2)}{F(r+d/2)}+\frac{o(d)}{d}\right\}\notag\\
      &\longrightarrow 1-2r\frac{F'\!(r)}{F\!(r)}\quad \text{as}\quad d\rightarrow 0.\notag 
     \end{align} 
     By Proposition \ref{geodproperty}, we have (\ref{main1pfeq}) for all $x,y\in X$. 
     This completes the proof of Theorem \ref{main1}.
   \end{proof}
   \begin{rem}[Heat kernel on Riemannian manifold]
   Let $p_t(x,y)$, $t>0$, $x,y\in M$, be the heat kernel on a Riemannian manifold $M$. 
   $p_t(x,\cdot)$ is 
   a fundamental solution of the heat equation $\partial_tu=\Delta u$, where $\Delta$ denotes 
   the Laplace-Beltrami operator. 
   A complete Riemannian manifold with Ricci curvature bounded by below always has 
   the heat kernel \cite{vRS}. 
   We set 
   $\nu_t=\nu*p_t=\left(\int_M p_t(x,\cdot)d\nu(x)\right)\!\mathrm{vol}$, where 
   $\mathrm{vol}$ denotes the Riemannian volume measure. We agree that $\nu_0=\nu$. 
   We call $\{\nu_t\}_{t\geq 0}$ 
   \emph{the heat distribution} or 
   \emph{heat flow} emanating from $\nu$.   
  The following is taught us by Nicola Gigli. 
   \begin{prop}[\cite{AGS2} p.44]
   \label{main3}
    Let $(M,g)$ be a complete Riemannian manifold 
    with Ricci curvature bounded below by $K$. Let $m^t_x$ be the heat distribution 
    emanating from the Dirac measure $\delta_x$ for $x\in M$.  
    Then the coarse Ricci curvature associated with $\{m^t_x\}_{x\in X}$ for any $t>0$ satisfies 
    \begin{align}
     \inf_{x,y\in M}\kappa(x,y) \geq 1-e^{-Kt}.\notag
    \end{align}
   \end{prop} 
   In \cite{AGS2}, they proved more general spaces, 
   called Riemann Ricci curvature bounded from below, have a lower bound of coarse Ricci curvature. 
   \end{rem}
  
   \begin{proof}[Proof of Theorem \ref{main2}]
    As (2) follows from (1), it suffices to show (1). 
    
    By the definition of the coarse Ricci curvature, a lower bound 
    $\kappa$ should satisfy $\kappa\leq 1$. Then we have $\kappa\leq 1\leq 2-\kappa$. 
    Accordingly we assume $\lambda\neq 1$. 
   Let $f$ be a non-constant Lipschitz function on $X$ such that $\Delta f=\lambda f$ 
   and $k$ the smallest Lipschitz constant of $f$. Replacing $f$ by 
   $(1/k)f$ if necessary, we assume that
    \begin{align}
    \label{lipconst}
     \sup_{x\neq y}\frac{|f(x)-f(y)|}{d(x,y)}=1.
    \end{align}
    By using (\ref{Kduality}), we obtain that
     \begin{align}
     \label{sym}
      d(x,y)(1-\kappa)&\geq W_1(m_x,m_y)\geq \int f\, dm_x-\int f\, dm_y\\
      &=f(x)-f(y)+\Delta f(y)-\Delta f(x)\notag\\
      &=(1-\lambda)(f(x)-f(y))\notag
     \end{align}
    for any $x,y\in X$. Since (\ref{sym}) is symmetric for $x$ and $y$, we have
     \begin{align}
      |1-\lambda|\cdot|f(x)-f(y)|&\leq d(x,y)(1-\kappa).\notag
     \end{align}
    By (\ref{sym}) and (\ref{lipconst}) it turns out that
     \begin{align}
      \frac{1-\kappa}{|\lambda-1|}\geq 1.\notag
     \end{align}
    Thus we get $\kappa\leq \lambda\leq 2-\kappa$. This completes the proof of Theorem\,\ref{main2}.
   \end{proof}
   To prove Corollary \ref{cor2}, it suffices to show the following Lemma. 
   \begin{lem}
   \label{lem3}
    Let $(X,d,\nu, \{m_x\})$ be as in Corollary \ref{cor2} and $v$ 
    an eigenfunction of the Laplacian 
    for the eigenvalue $\lambda\neq 1$. Then $v$ is a Lipschitz function. 
   \end{lem}
   \begin{proof}
    Setting $u:=(1-\lambda)v$, we have  
    \[
     u(x)=\int u\, dm_x=\frac{1}{\nu\left(B_{r}(x)\right)}\int_{B_r(x)}u\, d\nu.
    \]
    Then it suffices to prove that $u$ is a Lipschitz function.
    It is easy to see that $u$ is continuous. In fact, since $X$ is geodesic space, 
    \begin{align}
     &|u(x)-u(y)|=\left|\frac{1}{\nu\left(B_r(x)\right)}\int_{B_r(x)}u\, d\nu
     -\frac{1}{\nu\left(B_r(y)\right)}\int_{B_r(y)}u\, d\nu\right|\notag\\
     &=\left|\frac{1}{\nu\left(B_r(x)\right)}\left(\int_{B_r(x)}u\, d\nu-\int_{B_r(y)}u\, d\nu\right)
     +\left(\frac{1}{\nu\left(B_r(x)\right)}-\frac{1}{\nu\left(B_r(y)\right)}\right)\int_{B_r(y)}u\, d\nu\right|\notag\\
     &\leq \frac{1}{\nu\left(B_r(x)\right)}\int_{B_r(x)\triangle B_r(y)}|u|\, d\nu
     +\frac{\nu\left(B_r(x)\triangle B_r(y)\right)}{\nu\left(B_r(x)\right)\nu\left(B_r(y)\right)}
     \int_{B_r(y)}|u|\, d\nu\label{pflemeq1}\notag\\
     &\longrightarrow 0\qquad\text{as}\quad y\rightarrow x\notag
    \end{align}
    where $A\triangle B:=(A\setminus B)\cup (B\setminus A)$. 
    Since $B_r(x),B_r(y)\subset B_{2r}(x)$ for 
    $d(x,y)<r$ and $X$ satisfies a Bishop-Gromov inequality, we have 
    \begin{align}
     &|u(x)-u(y)|\leq \frac{1}{\nu\left(B_r(x)\right)}\int_{B_r(x)\triangle B_r(y)}|u|\, d\nu
     +\frac{\nu\left(B_r(x)\triangle B_r(y)\right)}{\nu\left(B_r(x)\right)\nu\left(B_r(y)\right)}
     \int_{B_r(y)}|u|\, d\nu\notag\\
     &\leq 2\sup_{B_{2r}(x)}|u|\cdot m_x\left(B_r(x)\triangle B_r(y)\right)\notag\\
     &=2\sup_{B_{2r}(x)}|u|\cdot m_x(B_r(x)\setminus B_r(y))\notag\\
     &\leq 2\sup_{B_{2r}(x)}|u|\left(\frac{F'(r-d(x,y)/2)}{F(r+d(x,y)/2)}d(x,y)+o(d(x,y))\right)\notag
    \end{align}
    which leads us to
    \[
     |u(x)-u(y)|\leq \mathcal{C}\,d(x,y)
    \]
    for any $y$ sufficiently close to $x$, where $\mathcal{C}=2(F'(r)/F(r)+1)\sup_{B_{2r}(x)}|u|$. 
    This means that $u$ is a local Lipschitz function. 
    Since $X$ is a compact metric space, $u$ is a Lipschitz function. 
   \end{proof}
   \begin{proof}[Proof of Corollary \ref{cor2}]
    By Lemma \ref{lem3}, any eigenfunction for the eigenvalue $\lambda\neq 1$ is Lipschitz. 
    Then we have $\kappa\leq\lambda\leq 2-\kappa$ for any $\lambda$ by Theorem \ref{main2}. 
    This completes the proof. 
   \end{proof}
   \begin{cor}
    Let $(X,d,\nu)$ be a compact geodesic metric measure space satisfying $[BG_{K,N}]$ for 
    two real numbers $K$ and $N>1$.  
    Then any eigenvalue $\lambda$ of the Laplacian associated with 
    the $r$-step random walk satisfies 
    \begin{align}
     -2r\frac{s_{K,N}(r)^{N-1}}{\int_0^r s_{K,N}(t)^{N-1}\, dt}\leq \lambda
     \leq 2+2r\frac{s_{K,N}(r)^{N-1}}{\int_0^r s_{K,N}(t)^{N-1}\, dt}.\notag
     \end{align} 
   \end{cor}
   \begin{proof}
    The corollary follows from Theorem \ref{main1} and Corollary \ref{cor2}. 
   \end{proof}
   \begin{cor}\label{corfinite}
    Let $X$ be a finite state space and $\{m_x\}_{x\in X}$ a random walk. 
    Assume that $\kappa(x,y)\geq \kappa$ for any $x,y\in X$ and for a constant $\kappa$. 
    Then any eigenvalue $\lambda$ of the Laplacian satisfies $\kappa\leq \lambda\leq 2-\kappa$.
   \end{cor}
   \begin{proof}
    Any function on a finite set is Lipschitz. We apply Theorem\,\ref{main2}.
   \end{proof}
   \begin{rem}
    Ollivier proved Corollary \ref{corfinite} if $\kappa>0$ and 
    if an unique invariant distribution is reversible. 
    We do not need these assumption. 
   \end{rem}
  \section*{Acknowledgement}
 The author is grateful to Professor Nicola Gigli for pointing out Proposition \ref{main3}, 
 Professor Shin-ichi Ohta for helpful comments and 
 Professor Takashi Shioya for reading this paper and giving useful advices. 
    

\begin{bibdiv}
\begin{biblist}

\bib{AGS2}{article}{
   author={Ambrosio, Luigi},
   author={Gigli, Nicola},
   author={Savar{\'e}, Giuseppe},
   title={Metric measure spaces with Riemannian Ricci curvature bounded from below},
   journal={arXiv:1109.0222},
}

\bib{AGS3}{article}{
   author={Ambrosio, Luigi},
   author={Gigli, Nicola},
   author={Savar{\'e}, Giuseppe},
   title={Calculus and heat flow in metric measure spaces and applications to spaces with Ricci bounds from below},
   journal={arXiv:1106.2090},
}





\bib{BJL}{article}{
    author={Frank Bauer, J{\"u}rgen Jost, and Shiping Liu},
    title={Ollivier-Ricci curvature and the spectrum of the normalized graph Laplace operator},
    journal={arXiv: 1105.3803v1},
}


\bib{E}{article}{
   author={Erbar, Matthias},
   title={The heat equation on manifolds as a gradient flow in the Wasserstein space},
   language={English, with English and French summaries},
   journal={Ann. Inst. Henri Poincar\'e Probab. Stat.},
   volume={46},
   date={2010},
   number={1},
   pages={1--23},
}

\bib{GKO}{article}{
   author={Gigli, Nicola},
   author={Kuwada, Kazumasa},
   author={Ohta, Shin-ichi},
   title={Heat flow on Alexandrov spaces},
   journal={arXiv:1008.1319},
}

\bib{JKO}{article}{
   author={Jordan, Richard},
   author={Kinderlehrer, David},
   author={Otto, Felix},
   title={The variational formulation of the Fokker-Planck equation},
   journal={SIAM J. Math. Anal},
   volume={29},
   date={1998},
   number={1},
   pages={1--17},
}

\bib{J}{article}{
   author={Juillet, Nicolas},
   title={Geometric inequalities and generalized Ricci bounds in the
   Heisenberg group},
   journal={Int. Math. Res. Not. IMRN},
   date={2009},
   number={13},
   pages={2347--2373},
}

\bib{JL}{article}{
    author={J{\"u}rgen Jost and Shiping Liu},
    title={Ollivier's Ricci curvature, local clustering and curvature dimension inequalities on graphs},
    journal={arXiv: 1103.4037v2},
}

\bib{LYY}{article} {
    author = {Lin, Yong},
    author={Yau, Shing-Tung},
    title = {Ricci curvature and eigenvalue estimate on locally finite
              graphs},
    journal = {Math. Res. Lett.},
    volume = {17},
    year = {2010},
    number = {2},
    pages = {343--356},
}

\bib{LV}{article}{
   author={Lott, John},
   author={Villani, C{\'e}dric},
   title={Ricci curvature for metric-measure spaces via optimal transport},
   journal={Ann. of Math. (2)},
   volume={169},
   date={2009},
   number={3},
   pages={903--991},
}

\bib{Oh}{article}{
   author={Ohta, Shin-ichi},
   title={On the measure contraction property of metric measure spaces},
   journal={Comment. Math. Helv.},
   volume={82},
   date={2007},
   number={4},
   pages={805--828},
}


\bib{O1}{article}{
   author={Ollivier, Yann},
   title={Ricci curvature of Markov chains on metric spaces},
   journal={J. Funct. Anal.},
   volume={256},
   date={2009},
   number={3},
   pages={810--864},
 
}


\bib{Sa}{article}{
   author={Savar{\'e}, Giuseppe},
   title={Gradient flows and diffusion semigroups in metric spaces under
   lower curvature bounds},
   language={English, with English and French summaries},
   journal={C. R. Math. Acad. Sci. Paris},
   volume={345},
   date={2007},
   number={3},
   pages={151--154},
}

\bib{St1}{article}{
   author={Sturm, Karl-Theodor},
   title={On the geometry of metric measure spaces. I},
   journal={Acta Math.},
   volume={196},
   date={2006},
   number={1},
   pages={65--131},
  
}

\bib{St2}{article}{
   author={Sturm, Karl-Theodor},
   title={On the geometry of metric measure spaces. II},
   journal={Acta Math.},
   volume={196},
   date={2006},
   number={1},
   pages={133--177},
}

\bib{V1}{book}{
   author={Villani, C{\'e}dric},
   title={Topics in optimal transportation},
   series={Graduate Studies in Mathematics},
   volume={58},
   publisher={American Mathematical Society},
   place={Providence, RI},
   date={2003},
   pages={xvi+370},
}

\bib{V2}{book}{
   author={Villani, C{\'e}dric},
   title={Optimal transport},
   series={Grundlehren der
 Mathematischen Wissenschaften [Fundamental
   Principles of Mathematical Sciences]},
   volume={338},
   note={Old and new},
   publisher={Springer-Verlag},
   place={Berlin},
   date={2009},
   pages={xxii+973},
}

\bib{vRS}{article}{
   author={von Renesse, Max-K.},
   author={Sturm, Karl-Theodor},
   title={Transport inequalities, gradient estimates, entropy, and Ricci
   curvature},
   journal={Comm. Pure Appl. Math.},
   volume={58},
   date={2005},
   number={7},
   pages={923--940},
}










\end{biblist}
\end{bibdiv}
\end{document}